\documentclass{amsart}
\usepackage{amsthm,amssymb}
\usepackage{verbatim}
\usepackage{hyperref}

\newtheorem{thm}{Theorem}
\newtheorem{prop}[thm]{Proposition}
\newtheorem{lem}[thm]{Lemma}
\newtheorem{conj}[thm]{Conjecture}

\theoremstyle{definition}
\newtheorem{defn}[thm]{Definition}

\newtheorem{rem}[thm]{Remark}

\newcommand{\ddb}{\sqrt{-1}\partial\bar{\partial}}
\renewcommand{\phi}{\varphi}
\renewcommand{\[}{\begin{equation}}
\renewcommand{\]}{\end{equation}}
\newcommand{\bigslant}[2]{{\left.\raisebox{.2em}{$#1$}\middle/\raisebox{-.2em}{$#2$}\right.}}

\begin{document}
\title{The J-flow and stability}
\author{Mehdi Lejmi}
\address{D\'epartement de Math\'ematiques, Universit\'e Libre de Bruxelles CP218, Boulevard du Triomphe, Bruxelles 1050, Belgique. }
\email{mlejmi@nd.edu}
\author{G\'abor Sz\'ekelyhidi}
\address{Department of Mathematics, University of Notre Dame, Notre
  Dame, IN 46615}
\email{gszekely@nd.edu}

\begin{abstract}
We study the J-flow from the point of view of an
algebro-geometric stability condition. In terms of this we give a
lower bound for the natural associated energy functional, and we show that
the blowup behavior found by Fang-Lai~\cite{FL12} is reflected by the
optimal destabilizer. Finally we prove a general existence result on
complex tori. 
\end{abstract}
\maketitle

\section{Introduction}
The J-flow was introduced by Donaldson~\cite{Don99} from the point of
view of moment maps, as well as Chen~\cite{Chen04} in his study of the
Mabuchi energy. To state the equation, let $(M,\omega)$ be a compact
K\"ahler manifold of dimension $n$, and let $\alpha$ be a second
K\"ahler metric on $M$ which is unrelated to $\omega$. The J-flow is
the parabolic equation
\[ \begin{aligned}
  \frac{\partial}{\partial t} \omega(t) &=
  -\ddb\Lambda_{\omega(t)}\alpha \\
  \omega(0) &= \omega,
\end{aligned}\]
where $\Lambda$ denotes the trace. The stationary solutions of this
flow are metrics $\omega$ such that
\[ \label{eq:Jeq} \Lambda_\omega\alpha = c, \]
where $c$ is a constant, which can be calculated from the K\"ahler
classes of $\omega$ and $\alpha$ using the equation
\[ \int_M \alpha\wedge\frac{\omega^{n-1}}{(n-1)!} = c\int_M
\frac{\omega^n}{n!}. \]
It was shown by Song-Weinkove~\cite{SW04} (see also
Weinkove~\cite{Wei04,W06})
that when a solution to
Equation~\eqref{eq:Jeq} exists, then the J-flow converges to this
solution. In \cite{SW04} the following necessary and sufficient condition was
given for the existence of a solution:
\begin{quotation} There exists a metric $\omega'\in [\omega]$ such that
  \[ c\omega'^{n-1} - (n-1)\omega'^{n-2}\wedge\alpha > 0, \]
  where the positivity means positivity of $(n-1,n-1)$-forms.
\end{quotation}
Unfortunately this condition is hard to check in concrete examples,
and it is not even clear whether the condition depends on the choice
of $\alpha$ in its K\"ahler class. In this paper we propose a new
numerical condition, which we conjecture is equivalent to existence of
a solution.

\begin{conj} \label{conj:main}
  A solution to Equation~\eqref{eq:Jeq} exists if and only
  if for all $p$-dimensional subvarieties $V\subset M$, where
  $p=1,2,\ldots,n-1$, we have
  \[ \label{eq:stab}
  \int_V c\omega^p - p\omega^{p-1}\wedge\alpha > 0. \]
\end{conj}

It is straightforward to show that this is indeed a necessary
condition. On the other hand one can also naturally arrive at this
condition from the point of view of an algebro-geometric stability
condition, analogous to K-stability, introduced by
Tian~\cite{Tian97}.
In fact, using that the
J-flow arises from a moment map, we can develop a theory parallel to
that of constant scalar curvature K\"ahler metrics and K-stability as
in Donaldson~\cite{Don02,Don05} for instance.

We will now focus on the situation when $\alpha$ and $\omega$
are algebraic, in the sense that they represent the first Chern classes
of ample line bundles. Suppose that $\omega\in c_1(L)$, and let
$M\hookrightarrow \mathbf{P}^N$ be an embedding using sections of
$L^k$ for some large $k$. A test-configuration $\chi$
for $(M,L)$ is obtained
by choosing a $\mathbf{C}^*$-action on $\mathbf{P}^N$, and we will
define an associated invariant $F_\alpha(\chi)$, analogous to the
Donaldson-Futaki invariant. Our first result, in
Section~\ref{sec:Bergman}
is a lower bound for a
natural energy functional in terms of this invariant, analogous to
Donaldson's lower bound for the Calabi functional~\cite{Don05}.

\begin{thm}\label{thm:lower} We have
  \[ \label{eq:lower}
     \inf_{\omega\in c_1(L)}
     \Vert \Lambda_\omega\alpha - c\Vert_{L^2} \geqslant
  \sup_{\chi\text{ test-config}}-\frac{F_\alpha(\chi)}{\Vert \chi\Vert}, \]
  where the $L^2$-norm is computed using $\omega$, and
  $\Vert\chi\Vert$ is a natural norm for test-configurations $\chi$
  for $(M,L)$. 
\end{thm}

A corollary of this result is analogous to Stoppa's
theorem~\cite{Sto08} on the K-stability of constant scalar curvature
K\"ahler manifolds.

\begin{thm}\label{thm:stab}
  If Equation~\eqref{eq:Jeq} has a solution, then
  $F_\alpha(\chi) > 0$ for all test-configurations $\chi$ for $(M,L)$ satisfying
  $\Vert\chi\Vert > 0$.
\end{thm}

In Section~\ref{sec:slope} we will study deformation to the normal
cone, which is a particular type of test-configuration studied
extensively by Ross-Thomas~\cite{RT06}. Applying
Theorem~\ref{thm:stab} to deformation to the normal cone of a
subvariety $V$ results in the inequality \eqref{eq:stab}. Note, however,
that this direction of Conjecture~\ref{conj:main} can easily be
checked directly. In this
section we will also show that Conjecture~\ref{conj:main} holds in the
two dimensional case. 

When Equation~\ref{eq:Jeq} has no solution, then it is natural to
conjecture that in Theorem~\ref{thm:lower} equality holds. We will
show this in a special case, namely on the blowup of $\mathbf{P}^3$ in
one point. The J-flow on this, and other similar, manifolds has been
studied carefully by Fang-Lai~\cite{FL12}. Our new contribution can be
summarized as follows.
\begin{thm}
On the blowup $\mathrm{Bl}_p\mathbf{P}^3$ in one point, with $\omega$
and $\alpha$ representing any two K\"ahler classes, equality
holds in Equation~\ref{eq:lower}. In addition the J-flow minimizes the
$L^2$-norm of $\Lambda_\omega\alpha$. 
\end{thm}
\noindent
This result is analogous to the second author's work~\cite{GSz07_1} on
the Calabi functional on a ruled surface.

In Section~\ref{sec:torus} we study Conjecture~\ref{conj:main} on
the complex torus $\mathbf{C}^n/(\mathbf{Z}^n+i\mathbf{Z}^n)$,
for $(S^1)^n$-invariant
data. It is easy to see that the inequalities \eqref{eq:stab} are
always satisfied, so we expect that a solution always exists.
In this situation Equation~\ref{eq:Jeq} reduces to an equation
for convex functions on $\mathbf{R}^n$. A generalization of this equation can
be formulated as follows. Let $a_{ij}(x)$ be a smooth $\mathbf{Z}^n$-periodic
symmetric positive definite matrix-valued function. We are trying to
find a $\mathbf{Z}^n$-periodic function $u:\mathbf{R}^n\to \mathbf{R}$
such that $f(x) = |x|^2 + u(x)$ is convex, and
\[\label{eq:torusjeq}
\sum_{i,j} a_{ij}(x) f^{ij}(x) = c,\]
where $f^{ij}$ is the inverse Hessian of $f$, and $c$ is a suitable
constant.
By using the Legendre transform and the
constant rank theorems of Korevaar-Lewis~\cite{KL87} and
Bian-Guan~\cite{BG09} we show that this equation always has a
solution.
\begin{thm}\label{thm:torusthm}
  Equation~\eqref{eq:torusjeq} has a smooth convex
  solution of the form
  $f(x) = |x|^2 + u(x)$, where $u$ is $\mathbf{Z}^n$-periodic.
\end{thm}

The J-flow can be generalized to the more general inverse $\sigma_k$-flows
studied by Fang-Lai~\cite{FL12}, and the results of this paper, apart
from Theorem~\ref{thm:torusthm}, extend
to this case without any difficulties. We will discuss this in
Section~\ref{sec:sigmak}.

\subsection*{Acknowledgements}
The first named author
is grateful to G\'abor Sz\'ekelyhidi and the rest of
Department of Mathematics at University of Notre Dame for their hospitality.
The second named author thanks Jeff Diller for several useful
conversations on the topic in the appendix. The
first named author was supported by a FQRNT grant during his visit to University
of Notre Dame. The second named author is
supported in part by NSF grant DMS-1306298.

\section{The Bergman kernel}\label{sec:Bergman}
The goal of this section is to prove Theorem~\ref{thm:lower}. The
method of proof is the same as Donaldson's proof of the analogous
lower bound for the Calabi functional~\cite{Don05}. The main
ingredient is a relevant Bergman kernel expansion.

Let $\omega\in c_1(L)$ for an ample line bundle $L$ over $M$. Choose a
Hermitian metric $h$ on $L$ such that $\omega = \frac{1}{2\pi i}
F(h)$, where $F(h)$ is the curvature form of $h$. On sections of $L^k$
define the inner product
\[\label{eq:inner} \langle s,t\rangle_{L^2} = k^n \int_M \langle s,t\rangle_{h^k}
\alpha \wedge \frac{\omega^{n-1}}{(n-1)!} = \int_M \langle
s,t\rangle_{h^k} (\Lambda_\omega\alpha)\,\frac{(k\omega)^n}{n!}.
\]
Given an orthonormal basis $\{s_0,\ldots,s_{N_k}\}$ of $H^0(L^k)$, define
the ``Bergman kernel''
\[ B_k(x) = \sum_{i=0}^N |s_i(x)|^2_{h^k}. \]
The following asymptotic expansion is analogous to the
Tian-Zelditch-Lu~\cite{Tian90_1,Zel98, Lu98} expansion of the usual
Bergman kernel associated to the metric $\omega$. 
\begin{thm}\label{thm:Bergman}
  We have the asymptotic expansion
  \[\label{eq:Bergman}
  B_k(x) = \frac{1}{\Lambda_\omega\alpha(x)} + O(k^{-1}), \]
  valid in $C^l$ for any $l$. 
\end{thm}
\begin{proof}
  The expansion follows from Theorem 4.1.1 in
  Ma-Marinescu~\cite{MM07}. Using the notation of \cite{MM07},
  we apply the result to $E$ being the
  trivial line bundle, with metric $\Lambda_\omega\alpha$. The
  endomorphism $\frac{1}{2\pi}\dot{R}^L$ is the identity, so in
  Theorem 4.1.1, equation (4.1.6) we have $\mathbf{b_0} = \mathrm{Id}_E$. In
  particular, by equation (4.1.4) in \cite{MM07} this means that
  \[ \sum_{i=0}^N |s_i(x)|^2_{h^k}(\Lambda_\omega\alpha) = 1 +
  O(k^{-1}), \]
  from which the required result follows. 
\end{proof}
Given an embedding $\phi:M\hookrightarrow\mathbf{P}^{N_k}$ we define
a matrix $M(\phi)$ with entries
\[ M(\phi)_{jk} = \int_M
\phi^*\left(\frac{Z^j\overline{Z}^k}{|Z|^2}\right)\alpha \wedge
\frac{(\phi^* \omega_{FS})^{n-1}}{(n-1)!}, \]
where $Z^j$ are homogeneous coordinates on $\mathbf{P}^{N_k}$ and
$\omega_{FS}$ is the Fubini-Study metric. Let $\underline{M}(\phi)$
denote the trace-free part of $M(\phi)$, so that
\[ \underline{M}(\phi)_{jk} = M(\phi)_{jk} -
\frac{k^{n-1}}{N_k+1}\int_M \alpha\wedge\frac{\omega^{n-1}}{(n-1)!}.\]
Similarly to Proposition 1 in \cite{Don05} we have the following. 
\begin{lem}
  There is a sequence of embeddings $\phi_k:M\hookrightarrow
  \mathbf{P}^{N_k}$ using sections of $L^k$ such that
  \[ \Vert \underline{M}(\phi_k)\Vert \leqslant k^{n/2-1}\Vert
  \Lambda_\omega\alpha - c\Vert_{L^2} + O(k^{n/2-2}). \]
  Here $\Vert \underline{M}\Vert = \mathrm{Tr}(\underline{M}^2)^{1/2}$. 
\end{lem}
\begin{proof}
  We use the sequence of embeddings $\phi_k$ defined by orthonormal bases of
  $H^0(L^k)$, with respect to the inner product \eqref{eq:inner}. We
  have
  \[ \phi_k^*\omega_{FS} = k\omega + \ddb B_k = k\omega + O(1), \]
  and so
  \[ \begin{aligned}
         M(\phi_k)_{ij} &= \int_M
         \phi_k^*\left(\frac{Z^i\overline{Z}^j}{|Z|^2}\right)
         \alpha\wedge \frac{(\phi_k^*\omega_{FS})^{n-1}}{(n-1)!} \\
         &= \int_M \frac{\langle s_i, s_j\rangle_{h^k}}{B_k}
         \alpha\wedge \frac{ (k\omega)^{n-1}}{(n-1)!} + O(k^{-2}).
       \end{aligned} \]
   By changing the basis of sections, we can assume that $M$ is
   diagonal. We have
   \[\begin{aligned}
     M(\phi_k)_{ii} &= k^{n-1}\int_M \frac{|s_i|^2_{h^k}}{B_k} \alpha
   \wedge \frac{\omega^{n-1}}{(n-1)!} + O(k^{-2}) \\
     &= k^{n-1}\int_M |s_i|^2_{h^k} (\Lambda_\omega\alpha)\alpha
     \wedge \frac{\omega^{n-1}}{(n-1)!} + O(k^{-2}),
   \end{aligned}\]
   and also the dimension of $H^0(L^k)$, by Riemann-Roch, is
   \[ N_k + 1 = k^n\int_M \frac{\omega^n}{n!} + O(k^{n-1}). \]
   It follows that
   \[ \begin{aligned}\sum_{i=0}^{N_k} M(\phi_k)_{ii} &= k^{n-1}\int_M B_k
   (\Lambda_\omega\alpha) \alpha\wedge \frac{\omega^{n-1}}{(n-1)!} +
   O(k^{n-2}) \\
   &= k^{n-1}\int_M (\Lambda_\omega\alpha)\,\frac{\omega^n}{n!} +
   O(k^{n-2}). 
  \end{aligned} \]
  The trace free part of $M$ is therefore
  \[ \underline{M}(\phi_k)_{ii} = k^{n-1}\int_M
  |s_i|^2_{h^k}(\Lambda_\omega\alpha -
  c)\,\alpha\wedge\frac{\omega^{n-1}}{(n-1)!} + O(k^{-2}). \]
  It follows that
  \[ \begin{aligned}
     \underline{M}(\phi_k)_{ii}^2 &= k^{2n-2} \left( \int_M
       |s_i|^2_{h^k}(\Lambda_\omega\alpha) (\Lambda_\omega\alpha -
       c)\,\frac{\omega^n}{n!}\right)^2 + O(k^{-3}) \\
     &\leqslant k^{2n-2}\int_M
     |s_i|^2_{h^k}(\Lambda_\omega\alpha)(\Lambda_\omega\alpha -
     c)^2\,\frac{\omega^n}{n!} \int_M |s_i|^2_{h^k}
   \Lambda_\omega\alpha\,\frac{\omega^n}{n!} + O(k^{-3}) \\
     &= k^{-2}\int_M |s_i|^2_{h^k}
     (\Lambda_\omega\alpha)(\Lambda_\omega\alpha -
     c)^2\,\frac{(k\omega)^n}{n!} + O(k^{-3}),
   \end{aligned} \]
   and so finally summing up over $i$ and using \eqref{eq:Bergman}, we
   have
   \[ \Vert \underline{M}(\phi_k)\Vert^2 \leqslant k^{-2}\int_M
   (\Lambda_\omega\alpha - c)^2\,\frac{(k\omega)^n}{n!} + O(k^{n-3}), \]
   from which the result follows. 
\end{proof}
We can obtain lower bounds for $\Vert \underline{M}(\phi_k)\Vert$
using test-configurations. For this, suppose that
$\lambda:\mathbf{C}^*\hookrightarrow GL(N_k+1)$ is a one-parameter
subgroup, such that $\lambda(S^1)\subset U(N_k+1)$. So
$\lambda(t)=t^A$ for a Hermitian matrix $A$ with integer
eigenvalues. A Hamiltonian function for the corresponding $S^1$-action
is given by
\[ h_A = \frac{A_{jk}Z^j\overline{Z}^k}{|Z|^2}. \]
Let $\phi_k^t = \lambda(t)\circ \phi_k$, and define the function
\[ f(t) = \mathrm{Tr}(\underline{A}\underline{M}(\phi_k^t)) =
\mathrm{Tr}(\underline{A}M(\phi_k^t)), \]
where $\underline{A}$ is the trace-free part of $A$. Then
\[ \label{eq:ft}
f(t) = \int_M \phi_k^{t*}(h_A)\,\alpha\wedge\frac{
  (\phi_k^{t*}\omega_{FS})^{n-1}}{(n-1)!} -
\frac{\mathrm{Tr}(A)k^{n-1}}{N_k+1} \int_M \alpha\wedge\frac{
  \omega^{n-1}}{(n-1)!}, \]
and a calculation shows that for real numbers $t>0$ we
have $f'(t)\geqslant 0$:
\begin{lem}
  With the above definition we have $f'(t)\geqslant 0$.
\end{lem}
\begin{proof}
  We consider the one-parameter
group of diffeomorphisms generated by the vector field
$-{\mathrm{grad}}\,h_A$ so we are
approaching $0$ along the positive real axis in $\mathbb{C}^\ast$.
Then, we have the following derivative at $s=0$
\[ \label{compute:ft}\begin{aligned}
\left. \frac{d}{ds}\right |_{s=0}\int_{M}  \phi_k^{s*}(h_A)\,\alpha\wedge\frac{
  (\phi_k^{s*}\omega_{FS})^{n-1}}{(n-1)!}&=-\int_{\phi_k(M)} |{\mathrm{grad}}\,h_A|^2\,{\phi_k}_\ast(\alpha)\wedge\frac{
  \omega_{FS}^{n-1}}{(n-1)!}\\
  &+\int_{\phi_k(M)} h_A\,{\phi_k}_\ast(\alpha)\wedge\frac{
  \mathfrak{L}_{-{\mathrm{grad}}\,h_A}\omega_{FS}\wedge\omega_{FS}^{n-2}}{(n-2)!}.
\end{aligned}
\]
The second term in the r.h.s of (\ref{compute:ft}) can be written as
\[ \begin{aligned}
\int_{\phi_k(M)} h_A\,{\phi_k}_\ast(\alpha)\wedge
  \mathfrak{L}_{-{\mathrm{grad}}\,h_A}\omega_{FS}\wedge\omega_{FS}^{n-2}&=2  \int_{\phi_k(M)} \partial h_A\wedge \bar\partial h_A\wedge{\phi_k}_\ast(\alpha)
 \wedge\omega_{FS}^{n-2}\\
 &=\frac{2}{n-1}\int_{\phi_k(M)} |{}\partial h_A|_M^2\,{\phi_k}_\ast(\alpha)\wedge\omega_{FS}^{n-1}\\&-\frac{2}{n(n-1)}\int_{\phi_k(M)} |{}\partial h_A|_{M,\alpha}^2\,\omega_{FS}^{n},
\end{aligned}
\]
where $|{}\partial h_A|_M^2=\frac{1}{2}|{\mathrm{grad}}\,h_A|_M^2$ is the norm of the tangential part to $\phi_k(M)$ and $|{}\partial h_A|_{M,\alpha}^2$ is
the norm with respect to ${\phi_k}_\ast(\alpha).$ We obtain 
\[ \begin{aligned}
\left. \frac{d}{ds}\right |_{s=0}\int_{M}  \phi_k^{s*}(h_A)\,\alpha\wedge\frac{
  (\phi_k^{s*}\omega_{FS})^{n-1}}{(n-1)!}&=-\int_{\phi_k(M)} |{\mathrm{grad}}\,h_A|_N^2\,{\phi_k}_\ast(\alpha)\wedge\frac{
  \omega_{FS}^{n-1}}{(n-1)!}\\&-2\int_{\phi_k(M)} |{}\partial h_A|_{M,\alpha}^2\,\frac{\omega_{FS}^{n}}{n!},
\end{aligned}
\] 
where $|{\mathrm{grad}}\,h_A|_N^2$ is the norm of the normal component. Increasing $t$
corresponds to flowing along ${\mathrm{grad}}\,h_A.$ We deduce that $f'(t)\geqslant 0$
for real numbers $t>0$.
\end{proof}
 
Now it follows that
\[ \mathrm{Tr}(\underline{A}\underline{M}(\phi_k)) = f(1) \geqslant
\lim_{t\to 0} f(t), \]
and so by the Cauchy-Schwarz inequality
\[ \Vert\underline{A}\Vert\, \Vert\underline{M}(\phi_k)\Vert \geqslant
\lim_{t\to 0}f(t). \]
In particular if $\lim_{t\to 0} f(t) > 0$, then we get a positive lower bound
on $\Vert \underline{M}(\phi_k)\Vert$.
We need to compute the limit on the right hand side.
\begin{lem}
Suppose that $\alpha\in c_1(K)$ for a very ample line bundle $K$ over
$M$, and let $D\subset M$ be a sufficiently general element in the
linear series defined by $K$. Let $D_0 = \lim_{t\to 0} \phi_k^t(D)$
denote the flat limit, and $|D_0|$ the corresponding algebraic cycle. Then
\[ \lim_{t\to 0} \int_M \phi_k^{t*}(h_A)\,\alpha\wedge \frac{
  (\phi_k^{t*}\omega_{FS})^{n-1}}{(n-1)!} = \int_{|D_0|}
h_A\,\frac{\omega_{FS}^{n-1}}{(n-1)!}. \]
\end{lem}
\begin{proof}
  According to Theorem~\ref{thm:averages} we can write $\alpha$ as a
  linear combination of currents of integrations:
  \[ \alpha = \int_{|K|} [D]\,d\mu(D), \]
  where $\mu$ is a smooth signed measure on the linear series
  $|K|$. It follows that
  \[ \begin{aligned}
    \int_M \phi_k^{t*}(h_A)\,\alpha\wedge \frac{
  (\phi_k^{t*}\omega_{FS})^{n-1}}{(n-1)!} &= \int_{|K|} \left( \int_D
  \phi_k^{t*}(h_A)\,
  \frac{(\phi_k^{t*}\omega_{FS})^{n-1}}{(n-1)!}\right)\,d\mu(D)\\
  &= \int_{|K|} \left(\int_{\phi_k^t(D)} h_A\,\frac{
      \omega_{FS}^{n-1}}{(n-1)!}\right)\,d\mu(D).
  \end{aligned}
  \]
  The integrands are uniformly bounded, so Lebesgue's convergence
  theorem implies that
  \[ \lim_{t\to 0} \int_M \phi_k^{t*}(h_A)\,\alpha\wedge \frac{
  (\phi_k^{t*}\omega_{FS})^{n-1}}{(n-1)!} = \int_{|K|} \left(
  \int_{|D_0|}
  h_A\,\frac{\omega_{FS}^{n-1}}{(n-1)!}\right)\,d\mu(D),\]
  where we abuse notation somewhat, denoting by $D_0=\lim_{t\to
    0}\phi^t_k(D)$ the flat limit, which depends on $D$. As in
  \cite{Sz12_1}, we now use the fact that the limit
  \[\label{eq:D0int} \int_{|D_0|} h_A\,\frac{\omega_{FS}^{n-1}}{(n-1)!} \]
  is related to the Chow weight of the divisor $D$ under the
  $\mathbf{C}^*$-action $\lambda$, to see that for all $D$ outside a
  Zariski closed subset of $|K|$, the value of the integral
  \eqref{eq:D0int} is the same. The result we want follows, since the
  Zariski closed subset has measure zero with respect to $\mu$.  
\end{proof}

Finally, just as in \cite{Don05} we can compute the asymptotics of
this lower bound on $\Vert\underline{M}(\phi_k)\Vert$ as
$k\to\infty$, when using the same test-configuration embedded into
larger and larger projective spaces.
This leads to the following definition.

\begin{defn}\label{defn:F}
 A test-configuration $\chi$ for $(M,L)$ consists of an embedding
 $M\subset \mathbf{P}^{N_k}$ of $M$ using
 sections of $L^k$, together with a one-parameter subgroup
 $\chi:\mathbf{C}^*\hookrightarrow GL(N_k+1)$. Let $M_0=\lim_{t\to 0}
 \chi(t)\cdot M$ denote the flat limit, with polarization $L_0^k$
 obtained by restricting the $O(1)$ bundle (so $L_0$ is a
 $\mathbf{Q}$-line bundle). There is an induced dual
 $\mathbf{C}^*$-action on $(M_0, L_0^m)$ for each sufficiently
 divisible $m$, and we let $d_m = \dim H^0(M_0, L_0^m)$, and denote by
 $w_m$ the total weight of the action. Define $a_0$ and $b_0$ using
 the expansions
 \[ \begin{aligned}
   d_m &= a_0m^n + O(m^{n-1}) \\
   w_m &=  b_0m^{n+1} + O(m^n)
\end{aligned}\]
As above, let $D\subset M$ be a sufficiently general element of the
linear series defined by $\alpha$ (we can replace $\alpha$ and
$\omega$ by multiples if necessary). The test-configuration $\chi$
induces a test-configuration for $D$, and we denote by $a_0'$ and
$b_0'$ the corresponding constants. Finally we define
\[ F_\alpha(\chi) = b_0' - \frac{a_0'}{a_0}b_0 = b_0' - cb_0, \]
where $c$ is the average of $\Lambda_\omega\alpha$ as above. 
The norm $\Vert\chi\Vert$ is defined exactly as in \cite{Don05} using
the asymptotics of $\mathrm{Tr}(A_m^2)$, where $A_m$ is the generator of
the $\mathbf{C}^*$-action on $H^0(M_0, L_0^m)$. 
\end{defn}

Note that from Proposition 3 in \cite{Don05}, it follows that if
$\lambda$ is a one-parameter subgroup of $GL(N_k+1)$ such that
$\lambda(S^1)\subset U(N_k+1)$ as before, then we have
\[ \int_{|D_0|} h_A\,\frac{\omega_{FS}^{n-1}}{(n-1)!} = -b_0' \]
in the notation of the above definition, and a similar equality holds
relating $b_0$ to the integral of $h_A$ over $|M_0|$ (note that the
eigenvalues of the dual action in the definition of $b_0$ are the
negatives of the eigenvalues of $\lambda$). It follows that
an alternative definition for $F_\alpha(\chi)$ is given by
\[ F_\alpha(\chi) = c\int_{|M_0|} h_A\,\frac{\omega_{FS}^n}{n!} -
\int_{|D_0|} h_A\,\frac{\omega_{FS}^{n-1}}{(n-1)!}, \]
where $|D_0|$ is the algebraic cycle representing the limit
$\lim_{t\to 0} \lambda(t)\cdot D$, for a generic divisor $D\subset M$
in the linear series defined by $\alpha$. Note that this is the same
expression that one needs to add to the usual Futaki invariant when
dealing with metrics with conical singularities along a divisor; see
Equation (30) in Donaldson~\cite{Don11_1}. 

From the arguments above, in an identical way to the proof in
\cite{Don05} we obtain the following.
\begin{thm}
We have
  \[ 
     \inf_{\omega\in c_1(L)}
     \Vert \Lambda_\omega\alpha - c\Vert_{L^2} \geqslant
  \sup_{\chi\text{ test-config}}-\frac{F_\alpha(\chi)}{\Vert \chi\Vert}, \]
  where the $L^2$-norm is computed using $\omega$.
\end{thm}

An immediate consequence of this result is that if a metric $\omega\in
c_1(L)$ exists for which $\Lambda_\omega\alpha=c$, then
$F_\alpha(\chi)\geqslant 0$ for all test-configurations $\chi$. We
now strengthen this to strict positivity, with a
perturbation method similar to Stoppa's work~\cite{Sto08} on
K-stability.
\begin{thm}
  If Equation~\eqref{eq:Jeq} has a solution, then
  $F_\alpha(\chi) > 0$ for all test-con\-fi\-gu\-ra\-tions $\chi$ for $(M,L)$ satisfying
  $\Vert\chi\Vert > 0$.
\end{thm}
\begin{proof}
  Suppose that there is a metric $\omega$ satisfying
  $\Lambda_\omega\alpha = c$, and let $\chi$ be a test-con\-fi\-gu\-ra\-tion
  for $(M,L)$. 
  For small $t \geqslant 0$ close to zero we write $\alpha_t = \alpha -
  t\omega$, which is still K\"ahler, and let
  \[ c_t = \frac{n\int_M \alpha_t\wedge \omega^{n-1}}{\int_M
    \omega^n}.\]
  A simple argument using the
  implicit function theorem (similar to that in the proof of
  Theorem~\ref{thm:torus1})
  implies that for $t$ sufficiently close to
  zero there is a metric $\omega_t$ satisfying
  $\Lambda_{\omega_t}\alpha_t=c_t$. The invariants $F_{\alpha_t}$ a
  priori are only defined for rational $t$, but they can be extended
  to all $t$ by continuity, and we have $F_{\alpha_t}(\chi) \geqslant
  0$ for $t$ sufficiently close to 0. In fact it is clear from the
  definition as the asymptotics of the limit of $f(t)$ in
  \eqref{eq:ft}, that $F_\alpha(\chi)$ is linear in $\alpha$, and so  
  \[ F_{\alpha_t}(\chi) = F_\alpha(\chi) - tF_\omega(\chi). \]
  We claim that $F_\omega(\chi) > 0$. If this is the case, then from
  $F_{\alpha_t}(\chi) \geqslant 0$ it follows that $F_\alpha(\chi) >
  0$. 

  For simplicity of notation
  assume that $\chi$ is given by a $\mathbf{C}^*$-action $\lambda$ on
  $\mathbf{P}^N$, and we have an embedding $M\subset \mathbf{P}^N$
  using sections of $H^0(M,L)$. Let $M_0$ denote the flat limit
  $\lim_{t\to 0} \lambda(t)\cdot M$. The assumption that
  $\Vert\chi\Vert > 0$ implies that the action of $\lambda$ is
  non-trivial on the reduced part of $M_0$, since the nilpotent
  structure gives rise to lower order terms in the expansion of
  $\mathrm{Tr}(A_k^2)$ in the definition of $\Vert \chi\Vert$ (see
  the discussion on p. 1406 of Stoppa~\cite{Sto08}). Alternatively,
  the norm is given by
  \[ \Vert\chi\Vert^2 = \int_{|M_0|} (h_A - \overline{h_A})^2\,
  \frac{\omega_{FS}^n}{n!}, \]
  where $\overline{h_A}$ denotes the average of $h_A$ on $|M_0|$, so
  the norm being positive implies that $h_A$ is non-constant on
  $|M_0|$.

  From the definitions it follows that to have $F_\omega(\chi) > 0$,
  we need
  \[ \int_{|M_0|} h_A\,\omega_{FS}^n > \int_{|D_0|} h_A\,
  \omega_{FS}^{n-1}, \]
  where $|D_0|$ denotes the limiting cycle of a generic divisor
  $D\subset M$ representing the class $[\omega]$ under the
  $\mathbf{C}^*$-action. In terms of Definition~\ref{defn:F} this
  means $b_0' > nb_0$. 

  In order to compute the induced
  test-configuration on $D$, it is useful to think of
  test-configurations as filtrations of the homogeneous coordinate
  ring as in \cite{Sz11} (see also \cite{Ny10}). Let
  \[ R = \bigoplus_{k\geqslant 0} H^0(M, L^k) = \bigoplus_{k\geqslant
    0} R_k \]
  be the homogeneous coordinate ring of $(M,L)$. As in \cite{Sz11},
  the test-configuration $\chi$ gives rise to a filtration
  \[ \mathbf{C} = F_0R \subset F_1R \subset F_2R \subset \ldots, \]
  where if necessary we multiply $\chi$ by an action with constant weights, to
  make all weights positive. Given a section $s\in R_1$, the divisor
  $D = (s=0)$ has homogeneous coordinate ring with $k^\mathrm{th}$
  graded piece $R_k \big / sR_{k-1}$, using the inclusion $sR_{k-1}\subset
  R_k$.  The filtration of $R$ induces a
  filtration on the coordinate ring of $D$ by letting
  \[ F_i\left( R_k \big / sR_{k-1}\right) = \left.\raisebox{.2em}{
      $F_i R_k$}\middle/\raisebox{-.2em}{$F_i(sR_{k-1})$}\right. \]
  The weight of the action on $R_k$ is given by
  \[ w_k = \sum_i (-i)\dim\bigslant{F_iR_k}{F_{i-1}R_k} , \]
  with a corresponding formula for the weight on
  $\bigslant{R_k}{sR_{k-1}}$:
  \[ \label{eq:wkprime}
  \begin{aligned} w_k' &= \sum_i (-i)\dim \bigslant{F_i(R_k\big /
      sR_{k-1})}{ F_{i-1}(R_k \big / sR_{k-1})} \\
    &= w_k - \sum_i (-i)\dim
    \bigslant{F_i(sR_{k-1})}{F_{i-1}(sR_{k-1})}.
    \end{aligned} \]
  In order to estimate the last sum, we will consider the central
  fiber of the test-configuration. The homogeneous coordinate ring of
  the central fiber $(M_0,L_0)$ can naturally be thought of as the
  associated graded ring of the filtration:
  \[ \widetilde{R} = \sum_i \bigslant{F_iR}{F_{i-1}R} = \sum_{k,i}
  \bigslant {F_iR_k}{F_{i-1}R_k}, \]
  with the induced $\mathbf{C}^*$-action acting on the $i^\mathrm{th}$
  piece with weight $-i$.  Let us write
  \[ s = s_1 + \ldots + s_m \]
  for the weight decomposition of the section $s$.

  We have
  \[ H^0(M_0, L_0^k) = \bigslant{\bigoplus\limits_j
    H^0(M_{0,j},L_0^k)}{E_k}, \]
  where $M_{0,j}$ are the irreducible components of
  $M_0$, and $E_k$ is a suitable subspace of the direct sum, defined
  by equality of sections on various intersections. Since these
  intersections are lower dimensional, $\dim E_k = O(k^{n-1})$, so to
  leading order in $k$, we can treat $H^0(M_0, L_0^k)$ as if it were
  equal to the direct sum. This way we can focus on each irreducible
  component separately. We need to check how multiplication
  by $s$ affects the weight of sections, and this will depend on
  which irreducible components the product does not vanish on.
  For each irreducible component $M_{0,j}$, let
  $m_j$ be the largest integer such that $s_{m_j}$ is
  not a zero divisor when restricted to $M_{0,j}$ (i.e. the reduced
  support $|M_{0,j}|$ is not contained in the zeroset of
  $s_{m_j}$). Then multiplication by $s$ will decrease the weight of
  sections which do not vanish on $M_{0,j}$ by at least $m_j$, and
  the
  total contribution of this is  $ (-m_j)\dim H^0(M_{0,j}, L_0^{k-1})$.
  Summing up over all irreducible components,
  some of the sections will be counted more than once, but up to order
  $k^{n-1}$  we get an upper bound
  \[ \label{eq:weights} \begin{aligned} \sum_i (-i) &\dim
  \bigslant{F_i(sR_{k-1})}{F_{i-1}(sR_{k-1})} \leqslant w_{k-1} +\\ &\qquad+ \sum_j
  (-m_j)\dim H^0(M_{0,j}, L_0^{k-1})+ O(k^{n-1}) \\
  &= w_{k-1} - \sum_j m_jk^n\int_{|M_{0,j}|}\frac{\omega_{FS}^n}{n!} + O(k^{n-1}).
  \end{aligned}\]
  If $s$ is a generic section, then it has non-zero component in each
  weight space for the $\mathbf{C}^*$-action $\lambda$, so for
  each $j$ we have that $m_j$ is the largest weight of the induced
  $\mathbf{C}^*$-action on the reduced support $|M_{0,j}|$. It follows
  that
  \[ \label{eq:M0j}
  \int_{|M_{0,j}|} h_A\,\frac{\omega_{FS}^n}{n!} \leqslant
  m_j\int_{|M_{0,j}|} \frac{\omega_{FS}^n}{n!}, \]
  and there is at least one $j$ such that $h_A$ is not
  constant on $|M_{0,j}|$, so that we have strict inequality in
  \eqref{eq:M0j}. Using this in \eqref{eq:weights} we get
  \[\sum_i (-i) \dim
  \bigslant{F_i(sR_{k-1})}{F_{i-1}(sR_{k-1})} < w_{k-1} -
  k^n\int_{|M_0|} h_A\,\frac{\omega_{FS}^n}{n!} + O(k^{n-1}). \]
  From \eqref{eq:wkprime} we then get
  \[ \begin{aligned} w_k' &> w_k - w_{k-1} + k^n\int_{|M_0|} h_A\,
    \frac{\omega_{FS}^n}{n!} + O(k^{n-1}) \\
    &= (n+1)b_0k^n -b_0k^n + O(k^{n-1}), 
    \end{aligned} \]
  so by taking leading order terms we have $b_0' > nb_0$. 
\end{proof}

\section{Deformation to the normal cone}\label{sec:slope}
A special type of test-configuration is given by deformation to the
normal cone of a subvariety. This was studied in detail by
Ross-Thomas~\cite{RT06} in the context of K-stability. Let $(M,L)$ be
a polarized manifold as before, and let $V\subset M$ be a
subvariety. The deformation to the normal cone of $V$ is the flat
family
\[ \mathrm{Bl}_{V\times\{0\}} M\times\mathbf{C} \]
over $\mathbf{C}$, obtained by blowing up. For sufficiently small
rational $\kappa > 0$ one can define the $\mathbf{Q}$-polarization
$\mathcal{L}_\kappa = \pi^*L - \kappa E$, where $E$ denotes the exceptional
divisor. Let us denote this test-configuration by $\chi_{V,\kappa}$. We can
compute the invariant $F_\alpha(\chi_{V,\kappa})$ using calculations from
\cite{RT06}, and from this we obtain the following.
\begin{prop}
  For sufficiently small $\kappa$ we have
  \[ F_\alpha(\chi_{V,\kappa}) = \frac{\kappa^{n-p+1}}{p!(n-p+1)!}
  \int_V c\omega^p - p\omega^{p-1}\wedge\alpha. \]
\end{prop}
\begin{proof}
  We need to compute the numbers $a_0, a_0', b_0, b_0'$ in
  Definition~\ref{defn:F}. First we have
  \[ \begin{aligned} a_0 &= \int_M \frac{\omega^n}{n!} \\
      a_0' &= \int_D \frac{\omega^{n-1}}{(n-1)!} = \int_M
      \alpha\,\wedge\frac{\omega^{n-1}}{(n-1)!},
    \end{aligned}\]
  where $D$ is any element of the linear series defined by $\alpha$. 
  To compute $b_0$, we use formula (4.6) from \cite{RT06}:
  \[ b_0 = \int_0^\kappa a_0(x)\,dx - \kappa a_0. \]
  For this we need to compute $a_0(x)$, defined by the expansion
  \[ \dim H^0(M, L^k\otimes \mathcal{I}_V^{xk}) = a_0(x)k^n + O(k^{n-1}), \]
  where $\mathcal{I}_V$ denotes the ideal sheaf of $V$. It follows
  that for small $x$ we have
  \[ a_0(x) = a_0 - \frac{x^{n-p}}{(n-p)!}\int_V
  \frac{\omega^p}{p!},\]
  where $p = \dim V$,  and so
  \[ b_0 = -\frac{\kappa^{n-p+1}}{(n-p+1)!}\int_V
  \frac{\omega^p}{p!}. \]
  To compute $b_0'$, note that for a generic $D$, the subvariety $V$
  has no component contained in $D$ and so the induced
  test-configuration for $D$ is deformation to the normal cone of
  $D\cap V$. The above formula can then be used in this case too and
  we obtain
  \[ b_0' = -\frac{\kappa^{n-p+1}}{(n-p+1)!}\int_{D\cap V}
  \frac{\omega^{p-1}}{(p-1)!} = -\frac{\kappa^{n-p+1}}{(n-p+1)!}\int_V
  \alpha\wedge\frac{\omega^{p-1}}{(p-1)!}. \]
  Our result then follows from the definition of
  $F_\alpha(\chi_{V,\kappa})$ in Definition~\ref{defn:F}. 
\end{proof}

Together with Theorem~\ref{thm:stab}, this result implies one
direction of Conjecture~\ref{conj:main}. Note that this direction also
follows directly by examining the eigenvalues of the metrics $\omega$
and $\alpha$ along $V$. Indeed, denoting by $\alpha_V$ and $\omega_V$
the restrictions to $V$, we have
\[ p\omega_V^{p-1}\wedge \alpha_V =
(\Lambda_{\omega_V}\alpha_V)\omega_V^p,\]
along $V$, and also $\Lambda_{\omega_V}\alpha_V < \Lambda_\omega\alpha
= c$. It follows that along $V$ we have
\[ c\omega_V^p - p\omega_V^{p-1}\wedge\alpha_V > 0, \]
and integrating we get \eqref{eq:stab}.

We should point out that part
of the content of Conjecture~\ref{conj:main} is that deformation to
the normal cone of subvarieties provides a sufficiently large family
of test-configurations to check, to ensure that a solution to
Equation~\ref{eq:Jeq} exists. We will show this in the case when $M$
is two-dimensional, however, it may be necessary to refine
the conjecture allowing for more general test-configurations in the
higher dimensional case. 

\begin{prop}
  When $\dim M = 2$, then Conjecture~\ref{conj:main} holds. 
\end{prop}
\begin{proof}
  We only need to show that if
  \[\label{eq:assumpt} \int_V c\omega - \alpha > 0 \]
  for all curves in $M$, then there exists a metric $\omega$
  satisfying $\Lambda_\omega\alpha = c$. According to 
  Chen~\cite{Chen04} it is enough to show that $[c\omega -
  \alpha]$ is a K\"ahler class.

  We will argue by contradiction, assuming that $[c\omega - \alpha]$
  is not K\"ahler with an argument similar to one in \cite{Don99}. 
  For any $t > 0$ let us define $\alpha_t = \alpha +
  t\omega$, and let
  \[ c_t = \frac{\int_M \alpha_t \wedge \omega}{\int_M
    \frac{1}{2}\omega^2}. \]
  Then
  \[ \label{eq:ddtct}
  \frac{d}{dt} (c_t\omega -\alpha_t) = 2\omega - \omega = \omega,\]
  so for sufficiently large $t$ the class $[c_t\omega - \alpha_t]$ is
  K\"ahler. Let
  \[ T = \inf\{t\,:\, [c_t\omega -\alpha_t]\text{ is K\"ahler}\}. \]
  Then $[c_T\omega - \alpha_T]$ is not K\"ahler, but it is
  nef, and in addition also big,
  since we can compute that
  \[ [c_T\omega - \alpha_T]^2 = [\alpha_T]^2 > 0. \]
  From the main result of Demailly-Paun~\cite{DP04} it follows that
  there is a curve $V\subset M$ such that
  \[ \int_V c_T\omega -\alpha_T = 0. \]
  But then we must also have
  \[ \int_V c\omega - \alpha \leqslant 0, \]
  because of \eqref{eq:ddtct}. This contradicts our assumption
  \eqref{eq:assumpt}.  
\end{proof}

\section{Example - $\mathbf{P}^3$ blown up in one point}
\label{sec:ruled}
In this section we will follow Fang-Lai~\cite{FL12} in studying a
concrete example, namely the blowup $M=\mathrm{Bl}_p\mathbf{P}^3$. The
discussion will be somewhat similar to the second author's
work~\cite{GSz07_1}, on the Calabi flow on a ruled surface. We write
$M=\mathbf{P}(\mathcal{O}(-1)\oplus\mathcal{O})$ as a ruled manifold
over $\mathbf{P}^2$. Let $h$ be a metric on $O(-1)$ with curvature
$-2\pi i\omega_{FS}$, and write $s = \log|\cdot|_h$ for the log of the
fiberwise norm. Given a suitably convex function
$f:\mathbf{R}\to\mathbf{R}$ we can write down a K\"ahler metric
\[ \alpha = \ddb f(s) \]
on $M$. At a point choose local coordinates $z=(z_1,z_2)$ on
$\mathbf{P}^2$ and a fiberwise coordinate $w$ such that
$d\log h(z)=0$. At this point we have
\[ \alpha = \sqrt{-1} f'(s) \pi^*\omega_{FS} + f''(s)
\frac{ \sqrt{-1} dw\wedge d\overline{w}}{2|w|^2}.\]
Similarly we can write
\[ \omega = \ddb g(s) = \sqrt{-1} g'(s) \pi^*\omega_{FS} + g''(s)
\frac{ \sqrt{-1} dw\wedge d\overline{w}}{2|w|^2}\]
for a different convex function $g$. It follows that
\[ \Lambda_\omega\alpha = 2\frac{f'(s)}{g'(s)} +
\frac{f''(s)}{g''(s)}.\]
Let us write $E_0$ for the zero section
(the exceptional divisor) on $M$ and $E_\infty$
the infinity section. Following \cite{FL12} we will work in the
K\"ahler classes
\[ \begin{aligned}
    \alpha &\in a[E_\infty] - [E_0] \\
    \omega &\in b[E_\infty] - [E_0],
  \end{aligned}
\]
for constants $a,b > 1$. In terms of $f, g$ this means that
\[ \begin{aligned}
  \lim_{s\to -\infty} f'(s) = 1,\quad \lim_{s\to \infty} f'(s) = a \\
  \lim_{s\to -\infty} g'(s) = 1, \quad\lim_{s\to \infty} g'(s) = b. 
\end{aligned}
\]
We introduce the coordinate $\tau = g'(s)$, and define the strictly
increasing function $F:[1,b] \to [1,a]$ by letting
\[ F(g'(s)) = f'(s) \]
for all $s$. We can then compute that in terms of $F$ we have
\[ \frac{dF}{d\tau}  + 2\frac{F}{\tau}, \]
and moreover if we think of $\alpha$ as being fixed, then we can
recover $\omega$ from knowing $F$. The main result of \cite{FL12} in
this special case is
that the J-flow on $M$ displays three different behaviors depending on
the values of $a,b$:
\begin{enumerate}
\item If $\displaystyle{\frac{ab^2-1}{b^3-1} > \frac{2}{3}}$, then the J-flow
  converges to a smooth solution of $\Lambda_\omega\alpha = c$.
\item If $\displaystyle{\frac{ab^2-1}{b^3-1} = \frac{2}{3}}$, then the J-flow
  converges to a singular solution of $\Lambda_\omega\alpha = c$,
  which is smooth away from $E_0$, and has a conical singularity along
  $E_0$.
\item
  If $\displaystyle{\frac{ab^2-1}{b^3-1} < \frac{2}{3}}$, then the J-flow converges
  to a current, which is a smooth solution of $\Lambda_\omega\alpha =
  c'$ (with a suitable constant $c'$) away from $E_0$, and is a current of
  integration along $E_0$. 
\end{enumerate}

In particular the equation
$\Lambda_\omega\alpha=c$ has a smooth solution on $M$, if and only if
\[ \label{eq:abineq}
\frac{ab^2-1}{b^3-1} > \frac{2}{3}.
\]
This is
consistent with Conjecture~\ref{conj:main}. Indeed we have
\[ \label{eq:Mc}
c = \frac{ 3 (a[E_\infty] - [E_0]) \cdot (b[E_\infty] - [E_0])^2}{
  (b[E_\infty] - [E_0])^3} = \frac{3(ab^2 - 1)}{b^3-1},\]
and so
\[ \int_{E_0} c\omega^2 - 2\omega\wedge\alpha = c - 2 = \frac{3(ab^2 -
  1)}{b^3-1} -2. \]
The latter quantity is positive precisely when the Inequality~\eqref{eq:abineq}
holds. In addition the fact that in case (2) and (3) the singularities
occur along $E_0$ is reflected by the fact that it is deformation to
the normal cone of $E_0$ which is the destabilizing test-configuration
in these cases. 

\begin{rem} Donaldson~\cite{Don99} pointed out that the obvious
  conjecture to make is that the J-flow converges whenever the class
  $[c\omega - \alpha]$ is K\"ahler. It is easy to check that on $M$,
  if we set $a = 5$ and $b=10$, then the latter class is K\"ahler, but
  Inequality~\ref{eq:abineq} does not hold. This means that the
  obvious conjecture is false.
\end{rem}

We will now use Theorem~\ref{thm:lower} to show that in these cases
the J-flow minimizes the $L^2$-norm of $(\Lambda_\omega\alpha-c)$. The
only interesting case is (3), since in the other two cases the infimum
is zero. As a consequence we will also see that equality holds in
Theorem~\ref{thm:lower}. In order to work with algebraic K\"ahler
metrics we need to assume that $a,b$ are rational, but a simple
approximation argument extends the results to arbitrary $a,b > 1$. 

\begin{thm}
  For any $a,b > 1$, the J-flow minimizes the $L^2$-norm of
  $\Lambda_\omega\alpha$. In addition equality holds in
  Equation~\eqref{eq:lower}. 
\end{thm}
\begin{proof}
  We will only focus on case (3). In \cite{FL12}, the J-flow is
  rewritten in terms of the function $F$, resulting in an evolution
  equation for a time dependent family of functions $F_t : [1,b] \to
  [1,a]$. It is then shown in \cite{FL12}, that as $t\to\infty$, the
  functions $F_t$ converge to $F_\infty$ satisfying
  \[ F_\infty(\tau) = \begin{cases}
    1,\quad & 1 \leqslant \tau \leqslant \lambda \\
    G(\tau),\quad & \lambda\leqslant\tau \leqslant b,
  \end{cases} \]
  for a suitable constant $\lambda \in (1, b)$, and $G$ satisfies the
  ODE
  \[ \frac{d}{d\tau}\left[\frac{dG}{d\tau}  +
    2\frac{G}{\tau}\right]=0, \]
  with boundary conditions $G(\lambda) = 1$, $G(b) = a$. This function
  $G$ is strictly increasing.

  The average of $\Lambda_\omega\alpha$ is the fixed constant $c$ in
  Equation~\ref{eq:Mc}, 
  which we can also compute from the function $F$. Namely, the
  volume form of $\omega$ is $\frac{1}{2}\tau^2\,d\tau$, so
  \[ \int_M \frac{\omega^3}{3!} = \int_1^b
  \frac{1}{2}\tau^2\,d\tau = \frac{b^3-1}{6}, \]
  and
  \[ \begin{aligned} \int_M \Lambda_\omega\alpha\, \frac{\omega^3}{3!} &= \int_1^b
  \left[ \frac{dF}{d\tau} + 2\frac{F}{\tau}\right]
  \frac{1}{2}\tau^2\,d\tau \\ &= \frac{1}{2} \int_1^b
  \frac{d}{d\tau}(\tau^2 F)\,d\tau \\ &= \frac{ ab^2 - 1}{2}, 
  \end{aligned} \]
  so the ratio of the two quantities recovers Equation~\eqref{eq:Mc}. 
  It is therefore equivalent to minimize $\Vert
  \Lambda_\omega\alpha\Vert_{L^2}$ or $\Vert\Lambda_\omega\alpha -
  c\Vert_{L^2}$. Using the work in \cite{FL12}, along the J-flow
  $\omega(t)$ we have
  \[ \label{eq:limit} \lim_{t\to\infty}
       \int_M (\Lambda_{\omega(t)}\alpha - c)^2\, \frac{\omega^3}{3!}
      = \int_1^b \left[ \frac{dF_\infty}{d\tau} +
        2\frac{F_\infty}{\tau} - c\right]^2
      \frac{1}{2}\tau^2\,d\tau. \]
  Our task is therefore to show that this limit is also a lower bound
  for the $L^2$-norm, using Theorem~\ref{thm:lower}. 

  Just as in \cite{GSz07_1}, we can write down test-configurations
  using piecewise linear, rational, convex functions on $[1,b]$. These
  are bundle versions of the toric test-configurations for the
  $\mathbf{P}^1$ fibers, studied by Donaldson~\cite{Don02}. We will
  use piecewise linear approximations to the convex function
  \[ h(\tau) = \frac{dF_\infty}{d\tau} + 2\frac{F_\infty}{\tau} - c
  = \begin{cases}
    2\tau^{-1} - c,\quad &1\leqslant\tau\leqslant\lambda \\
    2\lambda^{-1} - c, &\lambda\leqslant\tau\leqslant b. 
  \end{cases}
  \]

  We choose a sequence of piecewise linear, rational, convex functions
  $h_k$, approximating $h$. We can assume that each $h_k$ is constant
  for $\tau$ close to $b$. This means that each $h_k$ is a deformation
  to the normal cone of a suitable scheme supported on $E_0$. Let us denote this
  test-configuration by $\chi_k$.
  A general element in the linear series corresponding to $\alpha$ has
  no component contained in $E_0$. Let us write $F = \pi^*(O(1))$, so
  that
  \[ a[E_\infty] - [E_0] = [F] + (a - 1)[E_\infty]. \]
  In the limit along the central fiber of $\chi_k$, a
  generic element of the linear series corresponding to $\alpha$ will
  be the same as the induced test-configuration for a generic element
  of $[F]$ plus $(a-1)$-times $[E_\infty]$. Using this,  
  we can compute the numbers $b_{0,k}, b_{0,k}'$ in Definition~\ref{defn:F}
  for the invariant $F_\alpha(h_k)$, and we get
  \[ \begin{aligned}
    b_{0,k} &= -\int_1^b h_k\,\frac{1}{2}\tau^2\,d\tau \\
    b_{0,k}' &= -\int_1^b h_k \tau\, d\tau -
    \frac{(a-1)b^2}{2}h_k(b) \\
    \Vert\chi_k\Vert^2 &= \int_1^b h_k^2\,\frac{1}{2}\tau^2\,d\tau. 
  \end{aligned}\]
  To compute the right hand side of \eqref{eq:limit} we have
  \[ \begin{aligned}
    \int_1^b &\left[ \frac{d F_\infty}{d\tau} + 2\frac{F_\infty}{\tau}
    -c \right]^2\,\frac{1}{2}\tau^2\, d\tau = 
  \frac{1}{2}\int_1^b h \left[\tau^2\frac{dF_\infty}{d\tau} + 2\tau
    F_\infty -c\tau^2\right]\,d\tau\\
   &= \frac{1}{2}\int_1^b
  h\frac{d}{d\tau}( \tau^2F_\infty)\,d\tau - c\int_1^b
  h\frac{1}{2}\tau^2\,d\tau \\
  &= -\int_1^b \frac{dh}{d\tau} F_\infty\frac{1}{2}\tau^2\,d\tau +
  \frac{1}{2}\Big[b^2ah(b) - h(1)\Big] - c\int_1^b
  h\frac{1}{2}\tau^2\,d\tau \\
  &= -\int_1^b \frac{dh}{d\tau} \frac{1}{2}\tau^2\,d\tau +
  \frac{1}{2}\Big[b^2ah(b) - h(1)\Big] - c\int_1^b
  h\frac{1}{2}\tau^2\,d\tau \\
  &= \int_1^b h \tau\,d\tau + \frac{(a-1)b^2}{2}h(b) - c\int_1^b
  h\frac{1}{2}\tau^2\,d\tau \\
  &= \lim_{k\to\infty} \left[\int_1^b h_k\tau\,d\tau +
    \frac{(a-1)b^2}{2}h_k(b) -c\int_1^b
  h_k\frac{1}{2}\tau^2\,d\tau\right]  \\
  &= -\lim_{k\to\infty } F_\alpha(\chi_k),  
  \end{aligned}
  \]
  where in the fourth line we used that $h'(\tau)=0$ wherever
  $F_\infty(\tau) \not=1$.
  From Theorem~\ref{thm:lower} we obtain for any $\omega$ in our
  K\"ahler class the lower bound
  \[ \begin{aligned}
  \Vert \Lambda_\omega\alpha - c\Vert_{L^2} &\geqslant
  -\lim_{k\to\infty} \frac{F_\alpha(\chi_k)}{\Vert \chi_k\Vert} \\
  &= \frac{\int_1^b h^2\,\frac{1}{2}\tau^2\,d\tau}{\left(\int_1^b
      h^2\frac{1}{2}\tau^2\,d\tau\right)^{1/2}} \\
  &= \left(\int_1^b \left[ \frac{dF_\infty}{d\tau} +
      2\frac{F_\infty}{\tau} -
      c\right]^2\,\frac{1}{2}\tau^2\,d\tau\right)^{1/2} \\
  &= \lim_{t\to\infty} \Vert \Lambda_{\omega(t)}\alpha -
  c\Vert_{L^2},
  \end{aligned}
  \]
  where we used Equation~\ref{eq:limit} in the last line. This
  establishes that the J-flow minimizes the $L^2$-norm of
  $\Lambda_\omega\alpha$ as well as the fact that equality holds on
  Theorem~\ref{thm:lower} on the manifold $M$. 
\end{proof}

\section{Example - complex tori}\label{sec:torus}
In this section we will study the J-flow, or rather its critical
equation, on a  complex torus $M=\mathbf{C}^n /
(\mathbf{Z}^n+i\mathbf{Z}^n)$.
It is
easy to generalize to quotients by other lattices, so for simplicity
of notation we will focus on this specific case. The equation can then
be reduced to a special case of the following, in which $a_{jk}$ is
the Hessian of a function. It turns out that the stability condition
in this case is vacuous, which can also be seen from the fact that
``constant'' metrics in any two K\"ahler classes always provide 
solutions of the J-equation. 

Let $a_{jk}(x)$ be a smooth, symmetric positive definite matrix valued
function on $\mathbf{R}^n$, which is $\mathbf{Z}^n$-periodic. In
addition let $B = (b_{jk})$ be a symmetric positive definite
matrix. In terms of the J-equation, $a_{jk}(x)$ is determined by the
metric $\alpha$, and $B$ determines the K\"ahler class of $\omega$. 
\begin{thm}\label{thm:torus1}
  There exists a smooth convex function $f:\mathbf{R}^n\to\mathbf{R}$
  of the form
  \[ f(x) = \frac{1}{2}x^TBx + u(x), \]
  with $u(x)$ being $\mathbf{Z}^n$-periodic, that satisfies the
  equation
  \[ \label{eq:eq1} \sum_{j,k} a_{jk}(x) f^{jk}(x) = c.\]
  Here $f^{jk}$ is the inverse of the Hessian of $f$, and $c$ is a
  constant. The solution $f$ is unique up to addition of a constant. 
\end{thm}
\begin{proof}
  We argue using the continuity method, connecting $a_{jk}(x)$ to a
  constant matrix. When $a_{jk}(x)$ is constant, then a solution is
  given by $u(x)=0$. To prove openness is standard using the
  implicit function theorem. The linearization of the operator in
  \eqref{eq:eq1} at a solution $f$ is given by
\[ L:u \mapsto \sum_{j,k,p,q} a_{pq}(x) f^{jp}(x) f^{qk}(x)
\frac{\partial^2u(x)}{\partial x^j\partial x^k}, \]
 defined on periodic functions $u$. 
 This can be thought of as an elliptic operator on the
 torus. Deforming it into the Laplacian for the flat metric, we see
 that it has index zero, and moreover the strong maximum principle
 implies that any element of the kernel is constant. So the image of
 $L$ has codimension $1$. Moreover examining the maximum and miminum
 point of $u$ we see that non-zero constants are not in the image. It
 follows that
 \[ \begin{aligned}
          C^{k+2,\alpha} \times \mathbf{R} &\to C^{k,\alpha} \\
               (u, c) &\mapsto L(u) - c
\end{aligned}\]
is surjective. This is sufficient for openness.

The uniqueness also follows from the strong maximum principle. Namely,
suppose that $f$ and $h$ satisfy
\[ \sum_{j,k} a_{jk}(x)f^{jk}(x) = c_f, \quad \sum_{j,k}
a_{jk}(x)h^{jk}(x) = c_h. \]
Then, writing $f_t = h + t(f-h)$ for $t\in [0,1]$, 
\[ \begin{aligned}
  c_f - c_h &= \sum_{j,k} a_{jk}(x) (f^{jk}(x) - h^{jk}(x)) \\
  &= \sum_{j,k} a_{jk}(x) \int_0^1 \frac{d}{dt} f_t^{jk}(x)\,dt \\
  &= \sum_{j,k} a_{jk}(x) \int_0^1 \sum_{p,q} f_t^{jp}(x) (f-h)_{pq}(x)
  f_t^{qk}(x)\, dt \\
  &= \sum_{p,q} \left( \int_0^1 \sum_{j,k}
    f_t^{jp}(x)a_{jk}(x)f_t^{qk}(x)\,dt\right)
  \frac{\partial^2(f-h)}{\partial x^p\partial x^q}.
\end{aligned}\]
The coefficients in the brackets define a positive definite symmetric
matrix. Examining the maximum and minimum point of the periodic
function $f-h$ we find that
$c_f-c_h=0$, and then the strong maximum principle implies that $f-h$
is a constant. 

  It remains to prove a priori estimates for the solution. First of all
  it is easy to obtain a priori lower and upper bounds $0 < \underline{c}
  < c < \overline{c}$, in terms of the smallest and largest
  eigenvalues of $a_{jk}(x)$ and $B$ by examining the maximum and
  minimum points of $u$. Consider now the Legendre
  transform $g:\mathbf{R}^n \to\mathbf{R}$ of $f$, defined by
  \[ f(x) + g(y) = x\cdot y, \]
  where $y = \nabla f(x)$. It is standard that then $x=\nabla g(y)$,
  and
  \[ \left(\frac{\partial^2 f}{\partial x^j\partial
      x^k}\right)^{-1}(x) = \left(\frac{\partial^2 g}{\partial
      y^j\partial y^k}\right)(y). \]
  It follows that $g$ satisfies the equation
  \[ \label{eq:g} \sum_{j,k} a_{jk}(\nabla g(y)) \frac{\partial^2 g}{\partial
    y^j\partial y^k} = c. \]
  In addition $g$ is convex, so this equation implies a uniform upper
  bound on the Hessian of $g$ (in terms of the lowest eigenvalue of
  $a_{jk}(x)$). In particular we have a uniform $C^\alpha$ bound on
  $\nabla g$, and so Equation~\eqref{eq:g} implies uniform
  $C^{2,\alpha}$ bounds on $g$. Bootstrapping, we obtain uniform
  $C^{k,\alpha}$ bounds on $g$ for all $k$. The only question that
  remains is to find a positive lower bound on the Hessian of $g$,
  since that will then imply the required estimates on the Legendre
  transform $f$.

  We argue by contradiction. Suppose that there is a sequence of
  solutions $f_i$, with Legendre transforms $g_i$, such that the $g_i$
  do not have a uniform lower bound on their Hessians. The
  $C^{k,\alpha}$ bounds imply that we can find a subsequence
  converging in $C^\infty$ to a convex function $g_\infty$,
  solving an equation of the form
  \[\label{eq:ginf} \sum_{j,k} a^\infty_{jk}(\nabla g_\infty(y))
  \frac{\partial^2 g_\infty }{\partial
    y^j\partial y^k} = c_\infty > 0,  \]
  but with the Hessian of $g_\infty$ having a zero eigenvalue at some
  point. The constant rank result Corollary 1.3 in Bian-Guan~\cite{BG09} (see also
  Korevaar-Lewis~\cite{KL87} for the two-dimensional case) implies
  that the Hessian of $g_\infty$ is degenerate everywhere, and in particular
  there is a line $L\subset\mathbf{R}^n$, along which $g_\infty$ is
  linear. 
  This contradicts the fact that each $g_k$ (and so also $g_\infty$) is of the form
  \[ g_k(y) = \frac{1}{2}y^TB^{-1}y + v(y), \]
  where $v$ is periodic, with period $B\cdot\mathbf{Z}^n$. This
  implies that the solution $g$ of \eqref{eq:g} has a uniform lower
  bound on its Hessian (depending on bounds on $a_{jk}(x)$ and
  $B$). 
\end{proof}

\section{The inverse $\sigma_k$-flow}\label{sec:sigmak}
The discussion in the previous sections can be extended to a class of more
general equations introduced in Fang-Lai-Ma~\cite{FLM11},
called the inverse $\sigma_k$-flow. We are interested in the elliptic
version of the equation, which for $k=1,\ldots, n$ can be written as 
\[ \label{eq:sigmak}
\binom{n}{k}\alpha^k \wedge \omega^{n-k} = c \omega^n,
\]
where as before $\alpha$ is a fixed K\"ahler metric on the
$n$-dimensional K\"ahler manifold $M$, and we are trying to solve for
$\omega$ in a fixed K\"ahler class. In analogy with the result of
Song-Weinkove~\cite{SW04}, it is shown in \cite{FLM11} that a
necessary and sufficient condition for 
Equation~\eqref{eq:sigmak} to have a solution is the following:
\begin{quotation}
  There is a metric $\omega'\in [\omega]$ such that
  \[ c\omega'^{n-1} - \binom{n-1}{k}\omega'^{n-k-1}\wedge \alpha^k > 0,\]
  in the sense of positivity of $(n-1,n-1)$-forms.
\end{quotation}
In analogy with Conjecture~\ref{conj:main} one can make the following
conjecture.
\begin{conj}\label{conj:2}
  A solution of Equation~\ref{eq:sigmak} exists if and only if for all
  subvarieties $V\subset M$ of dimension $p$, for $p=k, k+1,\ldots,
  n-1$, we have
  \[ \int_V c\omega^p - \binom{p}{k}\omega^{p-k}\wedge\alpha^k > 0. \]
\end{conj}
Most of the results of the paper have suitable parallels in this
case, and we will briefly describe the modifications that need to be
made.

For the Bergman kernel expansion in Section~\ref{sec:Bergman}
we must use the inner product
\[ \langle s,t\rangle_{L^2} = k^n \int_M \langle s, t\rangle_{h^k}
\frac{\alpha^k \wedge \omega^{n-k}}{k! (n-k)!}. \]
To obtain the analogus result to Theorem~\ref{thm:lower} we define
$F_\alpha(\chi)$ as in Definition~\ref{defn:F}, but instead of letting
$D\subset M$ be a sufficiently general element of the linear series
defined by $\alpha$, we must take $D_1\cap \ldots\cap D_k\subset M$,
for $k$ sufficiently general elements. The coefficient $b_0'$ is
defined using the induced test-configuration for this intersection.

For deformation to the normal cone of a subvariety $V\subset M$, the
discussion is entirely analogous to that in
Section~\ref{sec:slope}. Note that if $\dim V < k$, then $V$ will be
disjoint from a generic intersection $D_1\cap \ldots \cap D_k$. This
is the reason why we only consider $V$ with $\dim V \geqslant k$ in
Conjecture~\ref{conj:2}. When $k=n$, then there is no condition at
all, and correspondingly in this case Equation~\eqref{eq:sigmak} is
just a prescribed volume form equation for $\omega$, which by Yau's
solution of the Calabi conjecture~\cite{Yau78} always has a solution. 

Analogous calculations can also be made to those in
Section~\ref{sec:ruled} following the work of
Fang-Lai~\cite{FL12}. The extension of Theorem~\ref{thm:torus1},
however, is more subtle, since one does not automatically obtain
$C^2$-bounds on the Legendre transform of the solution. We leave a
further study of this equation to future work.

\section{Appendix: smooth forms as averages of currents}
The goal of this section is to prove the following.
\begin{thm}\label{thm:averages}
  Suppose that $\alpha\in c_1(K)$, where $K$ is a very ample line
  bundle over $M$. Let $|K|$ denote the projective space of sections
  of $K$. There is a smooth signed measure $\mu$ on $|K|$ such that
  \[ \alpha = \int_{|K|} [D]\,d\mu(D), \]
  where $[D]$ denotes the current of integration along a divisor $D\in
  |K|$. The equality can be interpreted in the weak sense, i.e. for
  any smooth $(n-1,n-1)$-form $\beta$ we have
  \[ \int_M \alpha\wedge\beta = \int_{|K|}
  \left(\int_D\beta\right)\,d\mu(D). \]
\end{thm}
This result may be well-known to experts, but we have
not found this statement in the literature.
\begin{proof}
  We first prove the statement in the special case when
  $M=\mathbf{P}^n$ and $K=\mathcal{O}(1)$. Let us denote by
  $\mathbf{P}^{n^*}$ the dual projective space, and let $\mu_{FS}$ be
  the Fubini study volume form. Let $F$ be a smooth function on
  $\mathbf{P}^{n^*}$ with integral 1, and define 
  \[ \alpha_F = \int_{\mathbf{P}^{n^*}} [D] F(D)\,d\mu_{FS}(D). \]
  Then $\alpha_F\in c_1(\mathcal{O}(1))$, and we can compute
  $\alpha_F$ if we know its trace $\Lambda_{\omega_{FS}}\alpha_F$
  relative to the Fubini-Study metric. Indeed if
  \[ \alpha_F = \omega_{FS} + \ddb \phi, \]
  then
  \[ \Lambda_{\omega_{FS}}\alpha_F = n + \Delta_{\omega_{FS}}\phi, \]
  so $\phi$ can be recovered from the trace by solving the Poisson
  equation. We claim that at a point $p\in \mathbf{P}^n$ we have
  \[ \Lambda_{\omega_{FS}}\alpha_F(p) = \int_{H_p}
  F(D)\,d\mu_{FS}(D),\]
  where $H_p\subset\mathbf{P}^{n^*}$ denotes the hyperplanes passing through $p$, and
  $\mu_{FS}$ is the natural Fubini-Study measure on $H_p$. This
  equality follows from symmetry considerations, and the fact that
  $\alpha_F$ at the point $p$ depends only on the values of $F$ on
  $H_p$.

  In order to show that any $\alpha\in c_1(\mathcal{O}(1))$ is of the
  form $\alpha_F$, we simply need to show that any smooth function $G$
  on $\mathbf{P}^n$ can be written in the form
  \[ G(p) = \int_{H_p} F(D)\,d\mu_{FS}(D), \]
  for some smooth $F$ on $\mathbf{P}^{n^*}$. This transformation from
  $F$ to $G$ is a generalized Radon transform, and Theorem 4.1 in
  Helgason~\cite{Helg64} says that it is a one-to-one mapping. Our
  result therefore follows for the special case when
  $M=\mathbf{P}^n$.

  In the general case, use sections of $K$ to embed
  $\iota:M\hookrightarrow \mathbf{P}^N$, such that
  $K=\iota^*\mathcal{O}(1)$. We have
  \[ \alpha = \iota^*\omega_{FS} + \ddb \phi, \]
  for some smooth $\phi$, 
  so if $\widetilde{\phi}:\mathbf{P}^N\to\mathbf{R}$ denotes a smooth
  function such that $\iota^*\widetilde{\phi}=\phi$, then we have
  \[ \alpha = \iota^*(\omega_{FS} + \ddb \widetilde{\phi}) =
  \iota^*\widetilde{\alpha}, \]
  for a suitable $\widetilde{\alpha}\in c_1(\mathcal{O}(1))$. By the
  special case of our result, there is a smooth function
  $F:\mathbf{P}^{N^*}\to\mathbf{R}$ such that
  \[ \widetilde{\alpha} = \int_{\mathbf{P}^{N^*}} [D]
  F(D)\,d\mu_{FS}(D). \]
  Restricting to $\iota(M)\subset\mathbf{P}^N$, the result follows. 
\end{proof}

\providecommand{\bysame}{\leavevmode\hbox to3em{\hrulefill}\thinspace}
\providecommand{\MR}{\relax\ifhmode\unskip\space\fi MR }
\providecommand{\MRhref}[2]{%
  \href{http://www.ams.org/mathscinet-getitem?mr=#1}{#2}
}
\providecommand{\href}[2]{#2}

\end{document}